\documentclass{article}

\usepackage{amsmath,amsfonts,amssymb}

\usepackage{graphicx}

\usepackage{dsfont}
\usepackage{bbold}
\usepackage{amsthm}
\usepackage[utf8]{inputenc}
\usepackage{appendix}

\newtheorem{thm}{Theorem}

\newtheorem{defn}[thm]{Definition}
\newtheorem{prp}{Proposition}
\newtheorem{claim}{Claim}

\newtheorem{ques}{Question}

\title{\textbf{Perron's capacity of random sets}}

\author{Anthony Gauvan}

\begin{document}

\maketitle

\begin{abstract}
We answer in a probabilistic setting two questions raised by Stokolos in private communication. Precisely, given a sequence of random variables $\left\{ X_k : k \geq 1\right\}$ uniformly distributed in $(0,1)$ and independent, we consider the following random sets of directions $$\Omega_{\text{rand},\text{lin}} := \left\{  \frac{\pi X_k}{k}: k \geq 1\right\}$$ and $$\Omega_{\text{rand},\text{lac}} := \left\{ \frac{ \pi X_k}{2^k} : k\geq 1 \right\}.$$ We prove that almost surely the directional maximal operators associated to those sets of directions are not bounded on $L^p(\mathbb{R}^2)$ for any $1 < p < \infty$.
\end{abstract}

We denote by $\mathcal{R}$ the collection of all rectangles in the plane ; if $R$ belongs to $\mathcal{R}$, we denote by $\omega_R \in (0,\pi)$ the angle that its longest side makes with the $Oy$-axis. Without loss of generality, we will always suppose that we have actually $0 \leq \omega_R \leq \frac{\pi}{2}$.

\section{Introduction}

Given any set of directions $\Omega \subset \mathbb{S}^1$, one can define the directional family of rectangle $\mathcal{R}_\Omega$ as $$ \mathcal{R}_\Omega := \left\{ R \in \mathcal{R} : \omega_R \in \Omega \right\}$$ and then consider the \textit{directional maximal operator} $M_\Omega$ defined for $f : \mathbb{R}^2 \rightarrow \mathbb{R}$ locally integrable and $x \in \mathbb{R}^2$ as $$M_\Omega f(x) := \sup_{ x \in R \in \mathcal{R}_\Omega} \frac{1}{\left| R\right|}\int_R \left| f\right|.$$ The boundedness property of the operator $M_\Omega$ is deeply related to the geometric structure of the set of directions $\Omega$. For example, in the case where $\Omega = \mathbb{S}^1$, the following obstruction of the Euclidean plane (which is also true in higher dimension) allows us to completely describe the boundedness property of the operator $M_{\mathbb{S}^1}$.

\begin{thm}[Kakeya blow with $\mathcal{R}$]\label{THMKB}
Given any large constant $A \gg 1$, there exists a finite family of rectangles $\left\{ R_i : i \in I \right\} \subset {\mathcal{R}}$ such that we have $$\left| \bigcup_{i \in I} TR_i \right| \geq A\left| \bigcup_{i \in I} R_i \right|.$$ Here, we have denoted by $TR$ the rectangle $R$ translated along its longest side by its own length
\end{thm}

The reader can find a proof of this Theorem in \cite{FEFFERMAN}: it follows that given any large constant $A \gg 1$, there exists a bounded set $E$ satisfying the following estimate $$ \left| \left\{ M_{\mathbb{S}^1} \mathbb{1}_E \geq \frac{1}{2} \right\} \right| \geq A\left|E \right|.$$ It suffices to set $E = \cup_{i \in I} R_i$ and to observe that we have the following inclusion $$\bigcup_{i \in I} TR_i \subset \left\{ M_{\mathbb{S}^1} \mathbb{1}_E \geq \frac{1}{2} \right\}.$$ The previous estimate easily implies the following.

\begin{thm}
The operator $M_{\mathbb{S}^1}$ is not bounded on $L^p$ for any $p < \infty$.
\end{thm}

Far from being exotic, Theorem \ref{THMKB} has deep implications in harmonic analysis: for example, it is a central part of Fefferman's work in \cite{FEFFERMAN} where he disproves the famous \textit{Ball multiplier} conjecture. A natural question is the following: given a set of directions $\Omega$, is it possible to make a Kakeya blow only with the directional family $\mathcal{R}_\Omega$ ? This question have been investigated by different analyst among which \cite{DM}, \cite{GAUVAN}, \cite{KATZ}, \cite{NSW} or \cite{HareR} to cite a few. In \cite{BATEMAN}, Bateman answered this question as he classified the $L^p(\mathbb{R}^2)$ behavior of those operators according to the \textit{geometry} of the set $\Omega$. Precisely, he proved that the notion of \textit{finitely lacunary} for a set of directions were the correct one to consider.

\begin{thm}[Bateman]\label{T0}
We have the following alternative:
\begin{itemize}
    \item if $\Omega$ is finitely lacunary then $M_\Omega$ is bounded on $L^p$ for any $p > 1$.
    \item if $\Omega$ is not finitely lacunary then it is possible to make a Kakeya blow with the family $\mathcal{R}_\Omega$. In particular, the operator $M_\Omega$ is not bounded on $L^p$ for any $p < \infty$.
\end{itemize}
\end{thm}

Let us define the notion of \textit{finite lacunarity} following a nice presentation made by Kroc and Pramanik \cite{KP2}: we start by defining the notion of \textit{lacunary sequence} and then the notion of \textit{lacunary set of finite order}. We say that a sequence of real numbers $L=\left\{\ell_k : k \geq 1\right\}$ is a lacunary sequence converging to $\ell \in \mathbb{R}$ when there exists $0 < \lambda < 1$ such that $$|\ell-\ell_{k+1}| \leq \lambda|\ell - \ell_k|$$ for any $k$. For example the sequences $\left\{\frac{1}{2^k} : k \geq 2\right\}$ and $\left\{ \frac{1}{k!} : k\geq 4\right\}$ are lacunary. We define now by induction the notion of lacunary set of finite order.

\begin{defn}[Lacunary set of finite order]
A lacunary set of order $0$ in $\mathbb{R}$ is a set which is either empty or a singleton. Recursively, for $N \in \mathbb{N}^*$, we say that a set $\Omega$ included in $\mathbb{R}$ is a lacunary set of order at most $N+1$~---~and write $\Omega \in \Lambda(N+1)$~---~ wether there exists a lacunary sequence $L$ with the following properties : for any $a,b \in L$ such that $a < b$ and $(a,b) \cap L = \emptyset$, the set $\Omega \cap (a,b)$ is a lacunary set of order at most $N$ \textit{i.e.} $\Omega \cap (a,b) \in \Lambda(N)$.
\end{defn}

For example the set $$\Omega := \left\{ \frac{\pi}{2^k} + \frac{\pi}{4^l} : k,l \in \mathbb{N}, l \leq k \right\} $$ is a lacunary set of order $2$. In this case, observe that the set $\Omega$ cannot be re-written as a monotone sequence, since it has several points of accumulation. We can finally give a definition of a finitely lacunary set.

\begin{defn}[Finitely lacunary set]
A set $\Omega$ in $[0,\pi)$ is said to be \textit{finitely lacunary} if there exists a finite number of set $\Omega_1, \dots, \Omega_M$ which are lacunary of finite order such that $$ \Omega \subset \bigcup_{k \leq M } \Omega_k.$$
\end{defn}

\section{Can we apply Bateman's Theorem ?}

A classic example of set which is known to be not finitely lacunary is the set $$\Omega_{\text{lin}} = \left\{ \frac{\pi}{n} : n \in \mathbb{N}^* \right\}.$$ Indeed, the classic construction of \textit{Perron trees} shows that it is possible to make a Kakeya blow with the family $\mathcal{R}_{\Omega_{\text{lin}}}$: hence an application of Bateman's Theorem implies that $\Omega_{\text{lin}}$ is not finitely lacunary. The second classic example of set which is known to be finitely lacunary is the set $$\Omega_{\text{lac}} = \left\{ \frac{\pi}{2^n} : n \in \mathbb{N}^* \right\}.$$ One can see that the set $\Omega_{\text{lac}}$ is finitely lacunary by definition and it was in \cite{NSW} that Nagel, Stein and Waigner proved that the maximal operator $M_{\Omega_{\text{lac}}}$ is bounded on $L^p(\mathbb{R}^2)$ for any $1 < p < \infty$ ; their proof relies on Fourier analysis. \textit{Also, let us say that the comprehension of the sets $\Omega_{\text{lin}}$ and $\Omega_{\text{lac}}$ is important because they are the most simple (and smallest) cases of infinite sets which yield maximal operator having different boundedness properties.}

However, even if Bateman's Theorem is extremely satisfying, it appears to be difficult to decide if a given set $\Omega$ is finitely lacunary or not. The most striking example was raised by Stokolos : at the present time, it is not known if the set $$ \Omega_{\sin,\text{lac}} := \left\{ \frac{\pi \sin(n)}{2^n} : n \in \mathbb{N}^* \right\}$$ is finitely lacunary or not.  The main problem of this set of directions is that we have a very poor control on the deterministic sequence $\left\{ \sin(n) : n \geq 1 \right\}$ and that initially, the set $\Omega_{\text{lac}} $ is finitely lacunary: hence, the perturbations are quite difficult to handle. In \cite{DANIELLOGAUVANMOONENS}, with D'Aniello and Moonens, we were able to show that the following set  $$ \Omega_{\sin,\text{lin}} := \left\{ \frac{\pi \sin(n)}{n} : n \in \mathbb{N}^* \right\}$$ is not finitely lacunary (this was also a set considered by Stokolos). More precisely, we studied the maximal operator $M_{\Omega_{\sin,\text{lin}}}$ associated and, improving on concrete techniques, we proved that this operator is not bounded on $L^p(\mathbb{R}^2)$ for any $1 < p < \infty$: the heart of the method relied on the introduction of the \textit{Perron's capacity} of a set of directions. We need some notations to recall our results: given an infinite set of directions $\Omega \subset \mathbb{S}^1$ whose only point of accumulation is $0$ and we denote, for notational convenience, by $\Omega^{-1}$
the set $\left\{ \frac{\pi}{u} : u \in \Omega \right\}$, that is $$ \Omega^{-1} := \frac{\pi}{\Omega}$$ and we order $\Omega^{-1}$ as a strictly increasing sequence $\left\{  u_k:k\in\mathbb{N}^*  \right\}$. With those notations, we define the \textit{Perron's factor} of $\Omega$ as $$G(\Omega) := \sup_{\begin{subarray}{c}k \geq 1\\ l \leq k\end{subarray}} \left( \frac{u_{k+2l} - u_{k+l}}{u_{k+l} - u_{k}} + \frac{u_{k+l} - u_{k}}{u_{k+2l} - u_{k+l}} \right).$$ In \cite{HareR}, Hare and Ronning proved the following Theorem.

\begin{thm}[Hare and Ronning]\label{C2T1}
If we have $G(\Omega) < \infty$ then it is possible to make a Kakeya blow with the family $\mathcal{R}_\Omega$.
\end{thm}

It turns out that it is difficult to compute the Perron factor of the set $$\Omega_{\sin,\text{lin}} =  \left\{ \frac{\pi \sin(n)}{n} : n \geq 1 \right\}$$ since the oscillation of the cosinus prevent us to obtain a good description of the increasing sequence $\left\{  u_k:k\in\mathbb{N}^*  \right\}$ associated to $\Omega_{\sin,\text{lin}}$. Based on a careful read of the proof of Theorem \ref{C2T1}, for an arbitrary set of directions $\Omega$ included in $\mathbb{S}^1$, we define its \textit{Perron's capacity} as $$P(\Omega) :=\liminf_{N \to\infty} \inf_{ \begin{subarray}{c}U \subset \Omega^{-1} \\ \#U = 2^N \end{subarray}}  G(\Omega) \in [2,\infty]$$ where as before $$G(\Omega) = \sup_{\begin{subarray}{c}k,l \geq 1\\ k+2l \leq 2^N\end{subarray}} \left( \frac{u_{k+2l} - u_{k+l}}{u_{k+l} - u_{k}} + \frac{u_{k+l} - u_{k}}{u_{k+2l} - u_{k+l}} \right)$$ if $U = \{ u_1 < \dots < u_{2^N} \}$. In \cite{DANIELLOGAUVANMOONENS}, we proved the following (in contrast with Hare and Ronning Theorem, we do not assume that the set $\Omega$ is ordered):

\begin{thm}[D'Aniello, G. and Moonens]\label{T6}
For any set of directions $\Omega$, if we have $$P(\Omega) < \infty$$ then it is possible to make a Kakeya blow with the family $\mathcal{R}_\Omega$. In particular, for any $p < \infty$, one has $ \left\|M_\Omega \right\|_p = \infty$.
\end{thm}

Loosely speaking, if $P(\Omega) < \infty$ then the set $\Omega$ contains arbitrary large set which are (more or less) uniformly distributed and this geometric pattern prevents the set $\Omega$ to be finitely lacunary. The advantage of Theorem \ref{T6} is that it allows us to make \textit{effective} computation. However, as mentioned earlier, the following case is still unsettled.

\begin{ques}
Is the following set of direction $$ \Omega_{\sin,\text{lac}} := \left\{ \frac{\pi\sin(n)}{2^n} : n \in \mathbb{N}^* \right\}$$ finitely lacunary or not ?
\end{ques}

\section{Results}

Our result concerns random sets of directions which are meant to give a \textit{generic} comprehension of the two classic examples $\Omega_{\text{lin}}$ and $\Omega_{\text{lac}}$ when they are randomly perturbated. Precisely we consider the following random sets of slopes $$\Omega_{\text{rand},\text{lin}} := \left\{ \frac{\pi X_k}{k} : k \geq 1 \right\} $$ and $$\Omega_{\text{rand},\text{lac}} := \left\{ \frac{\pi X_k}{2^k} : k \geq 1 \right\} $$ where $\left\{ X_k : k \geq 1 \right\}$ are random variables uniformly distributed in $(0,1)$ and independent. To begin with, we prove the following Theorem.

\begin{thm}\label{T1}
The Perron's capacities of $\Omega_{\text{rand},\text{lin}}$ is finite almost surely \textit{i.e.} we have almoste surely $$P(\Omega_{\text{rand},\text{lin}}) < \infty.$$
\end{thm}

In some sense, Theorem \ref{T1} means that if a set $\Omega$ presents structured patterns - like large uniformly distributed sequence - then a small perturbation of $\Omega$ will still exhibit those patterns. The second result reads as follow.

\begin{thm}\label{T2}
The Perron's capacities of $\Omega_{\text{rand},\text{lac}}$ is finite almost surely \textit{i.e.} we have almoste surely $$P(\Omega_{\text{rand},\text{lac}}) < \infty.$$
\end{thm}

The proof of Theorems \ref{T1} and \ref{T2} relies on the possibility to compute effectively the Perron's capacity of the random sets $\Omega_{\text{rand},\text{lin}}$ and $\Omega_{\text{rand},\text{lac}}$.

\section{Proof of Theorem \ref{T1}}

We wish to prove that the Perron's capacity of $\Omega_{\text{rand},\text{lin}}$ is finite almost surely. We are simply going to prove that \textit{the set $\Omega^{-1}_{\text{rand},\text{lin}}$ contains small perturbation of arbitrarily long homogeneous sets.} We say that a set $H$ of the form $$H := H_{a,N} = \{ ka :1 \leq  k \leq 2^N \}$$ for some integer $a \in \mathbb{N}^*$ is an \textit{homogeneous set}. The following claim is easy.

\begin{claim}
For any $a,N \in \mathbb{N}$, one has $G(H_{a,N}) = 2$.
\end{claim}

We wish to perturb a little an homogeneous set $H$ into a set $H'$ such that the Perron's factor of $H'$ is still controlled. Precisely, fix any $a,N\in \mathbb{N}^*$ and let $\epsilon$ be an arbitrary function $$ \epsilon : H_{a,N} \rightarrow (0,\infty).$$ Define then the set $H_{a,N}(\epsilon)$ as $$H' := H_{a,N}(\epsilon) := \left\{ (1+\epsilon(l)) l : l \in H_{a,N} \right\}.$$ If the perturbation $\epsilon$ is small enough compared to the integer $N$, one can control uniformly $G(H')$.

\begin{prp}\label{P5}
With the previous notations, if we have $$2^N\|\epsilon\|_\infty \leq \frac{1}{2}$$ then we have $G(H_{a,N}(\epsilon)) < 6$.
\end{prp}

We are now ready to prove Theorem \ref{T1}. We fix a large integer $N \in \mathbb{N}$ and consider the following set of indices $$E_N := \left\{ k \in \mathbb{N} : |X_k-1| \leq 2^{-N} \right\}.$$ In other words, an integer $k$ belongs to $E_N$ when $X_k$ is close to $1$ with precision $2^{-N}$. We claim that this random set $E_N$ contains almost surely large (with at least $2^N$ points) homogeneous sequences.

\begin{claim}
For any $N \geq 1$, the set $E_N$ contains an homogeneous set of cardinal $2^N$ almost surely.
\end{claim}

\begin{figure}[h!]
\centering
\includegraphics[scale=0.7]{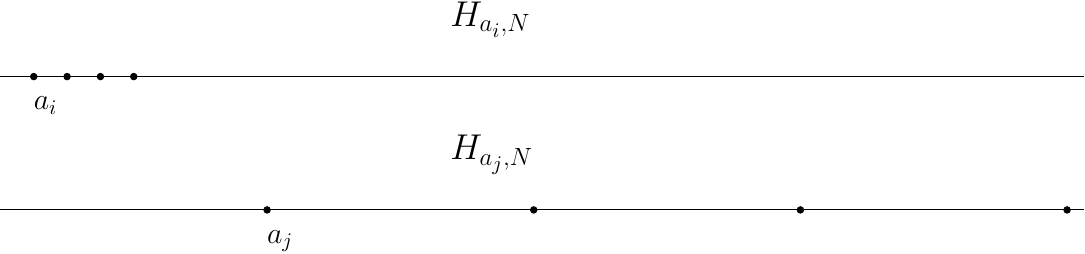}
\caption{We choose the sequence $\{ a_i : i \geq 1 \}$ such that $H_{a_i,N} \cap H_{a_j,N} = \emptyset$ for any $i \neq j$.}  
\end{figure}

\begin{proof}
Observe that for any $a \in \mathbb{N}^*$, the following probability $\mathbb{P}( H_{a,N} \subset E_N)$ is independent of $a$: indeed since the random variables $\left\{ X_k : k \geq 1 \right\}$ are independent and uniformly distributed, we have $$\mathbb{P}( H_{a,N} \subset E_N) = \prod_{k \in H_{a,N}}\mathbb{P}\left(|X_k-1| \leq 2^{-N}\right) = \frac{1}{2^{N2^N}} .$$ Hence we fix a sequence $\left\{ a_i \right\}_{i \geq 1}$ satisfying for any $i\neq j$ $$H_{a_i,N} \cap H_{a_j,N} = \emptyset.$$ For example, setting $a_i = 2^{2N(i+1)}$ works since we have $a_i 2^N < a_{i+1}$ for any $i \geq 1$. In particular this means that the events $$\{ \left(H_{a_i,N} \subset E_N\right) : i \geq 1 \}$$ are independent and since we have $$\sum_{i \geq 1} \mathbb{P}( H_{a_i,N} \subset E_N) = \infty $$ an application of Borel-Cantelli lemma yields the conclusion. 
\end{proof}

We can now conclude the proof : we define a perturbation $\epsilon$ for any $n \geq 1$ as $$1 + \epsilon(n) = X_n^{-1}.$$ We fix a large integer $N \gg 1$ and we know that almost surely there exists an integer $a \in \mathbb{N}^*$ such that $H_{a,N} \subset E_N$. Observe now that by definition one has the following inclusion $$H_{a,N}(\epsilon) \subset \Omega^{-1}_{\text{rand},\text{lin}}.$$ However since $H_{a,N} \subset E_N $ and that it is not difficult to see that we have $$ \|\epsilon_{|H_{a,N}} \|_\infty \lesssim 2^{-N}.$$ Indeed, for $k \in H_{a,N} \subset E_N$, we have $$\left|1 + \epsilon(k) \right| = \left| \frac{1}{1 +(X_k-1)} \right| \lesssim 1 + 2\left|X_k-1\right| \lesssim 1 + 2^{-N}.$$ It follows that we have $P(\Omega_{\text{rand},\text{lin}}) < 6$ almost surely applying Proposition \ref{P5}.

\section{Proof of Theorem \ref{T2}}

We wish to prove that the Perron's capacity of $\Omega_{\text{rand},\text{lac}}$ is finite almost surely. Observe that if $U$ is a set in $\mathbb{R}$ who is well distributed then one can control its Perron's capacity.

\begin{claim}\label{P2}
If we have $\delta > 0$ and a set $$U = \left\{u_1 < \dots < u_{2^N} \right\} $$ such that for any $1 \leq i \leq 2^N-1$ we have $$\delta \leq u_{i+1} - u_i \leq 3 \delta$$ then one has $G(U) \lesssim 1$.
\end{claim}

\begin{proof}
For any $i,j$ such that $i+j \leq 2^N-1$, one has $u_{i+j} - u_i \simeq j\delta$. Hence $i,j$ such that $i+2j \leq 2^{N-1}$, we have $$\frac{u_{i+2j} - u_{i+j}}{u_{i+j} - u_{i}} + \frac{u_{i+j} - u_{i}}{u_{i+2j} - u_{i+j}} \simeq \frac{2j\delta}{j\delta} + \frac{j\delta}{2j\delta} \simeq 1.$$ Hence we obtain $G(U) \lesssim 1$ as claimed
\end{proof}

We are going to prove the following proposition.

\begin{prp}\label{P3}
For any $N \geq 1$, there exists almost surely a scale $\delta > 0$ and a set $U \subset \Omega^{-1}_{\text{rand},\text{lac}}$ such that $$U := \left\{ u_1 < \dots < u_{2^{N-1}} \right\} $$ and for any $i \leq 2^{N-1}$ one has $$\delta < u_{i+1}-u_i < 3\delta.$$
\end{prp}

Theorem \ref{T2} is a consequence of Claim \ref{P2} and Proposition \ref{P3}: for any $N$, we can exhibit almost surely a set $U \subset \Omega^{-1}_{\text{rand},\text{lac}}$ of cardinal $2^N$ such that $G(U) \lesssim 1$ and so we obtain $$ P(\Omega_{\text{rand},\text{lac}}) < \infty$$ as expected. The rest of the section is devoted to the proof of Proposition \ref{P3}.

\subsection*{Proof of Proposition \ref{P3}}

We consider the following dyadic intervals for $d \in \mathbb{N}$ $$I_d := \left[2^d ,2^{d+1}\right].$$ We wish to obtain information on the distribution of the points of the set $\Omega_{\text{rand},\text{lac}}$ that may be in the interval $I_d$. We fix a large integer $N \gg 1$ and we divide each dyadic interval $I_d$ into $2^N$ intervals of same length \textit{i.e.} for any $1 \leq l \leq 2^N$ we set $$I_{d,l} = \left[2^d\left(1 + \frac{l-1}{2^N} \right), 2^d\left(1 + \frac{l}{2^N} \right) \right].$$

\begin{figure}[h!]
\centering
\includegraphics[scale=0.7]{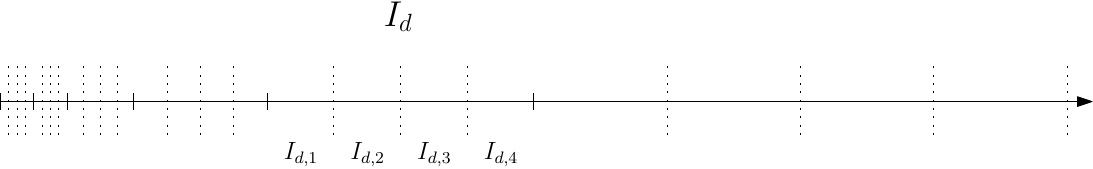}
\caption{For $N = 2$, each intervals $I_d$ is divided in $2^N = 4$ equal parts.}  
\end{figure}

\begin{claim}\label{CCC}
For any $d \geq 2^N+1$ and any $1 \leq l \leq 2^N$, the probability $$\mathbb{P}\left( \frac{2^{d-l}}{X_{d-l}} \in I_{d,l} \right) := p_{N,l}$$ is independent of $d$.
\end{claim}

\begin{proof}
By definition of $I_{d,l}$, one has $2^{d-l}X_{d-l}^{-1}   \in I_{d,l}$ if and only if $$ 2^{-l}\left(1 + \frac{l}{2^N} \right)^{-1} \leq X_{d-l} \leq 2^{-l}\left(1 + \frac{l-1}{2^N} \right)^{-1}.$$ One has $$\mathbb{P}\left( 2^{d-l}X_{d-l}^{-1} \in I_{d,l} \right) = 2^{-l}\left( \left( 1 + \frac{l-1}{2^N} \right)^{-1} - \left(1 + \frac{l}{2^N} \right)^{-1}  \right) := p_{l,N}$$ since the variable $X_{d-l}$ is uniformly distributed in $(0,1)$.
\end{proof}

We fix an extraction $\{d_s : s \geq 1 \}$ satisfying the following property $$d_{s+1} - d_s > 2^N + 1 $$ for any $s \geq 1$ ; this growth condition will assure that the events we will consider are independent and we will be able to apply Borel-Cantelli lemma when needed. Thanks to Claim \ref{CCC} and independence, one can see that for any $s \geq 1$, the following probability is independent of $s$ $$ \mathbb{P}\left( \bigcap_{l \leq 2^N} \left(\frac{2^{d_s-l}}{X_{d_s-l}}  \in I_{d_s,l}\right) \right) =  \prod_{l \leq 2^N} p_{N,l} := \eta_N $$ Now for any $d \geq 1$, we consider the following event $A_{d,N}$ defined as $$A_{d,N} :=  \bigcap_{l \leq 2^N} \left\{ \Omega^{-1}_{\text{rand},\text{lac}} \cap I_{d,l} \neq \emptyset \right\}.$$ In other words, the event $A_{d,N}$ occurs when the random set $\Omega^{-1}_{\text{rand},\text{lac}}$ fills each sub-intervals $\left\{ I_{d,l} : 1 \leq l\leq 2^N \right\}$ with at least one point. In particular, observe that we have $$ \bigcap_{l \leq 2^N} \left(\frac{2^{d-l}}{X_{d-l}}  \in I_{d,l}\right) \subset A_{d,N}.$$ We claim that the union of those events $$B_N := \bigcup_{d \geq 1} A_{d,N}$$ occurs almost surely.

\begin{claim}
One has $\mathbb{P}\left(B_N\right) = 1$.
\end{claim}

\begin{proof}
Indeed we have $$\sum_s  \mathbb{P}\left( \bigcap_{l \leq 2^N} \left(\frac{2^{d_s-l}}{X_{d_s-l}}  \in I_{d_s,l}\right) \right) = \sum_s \eta_N = \infty.$$ Using Borel-Cantelli lemma, one obtains $$\mathbb{P}\left( \bigcup_{s \geq 1} A_{N,d_s} \right) = 1.$$ In particular, one has $\mathbb{P}(B_N) = 1$.
\end{proof}

\begin{figure}[h!]
\centering
\includegraphics[scale=0.7]{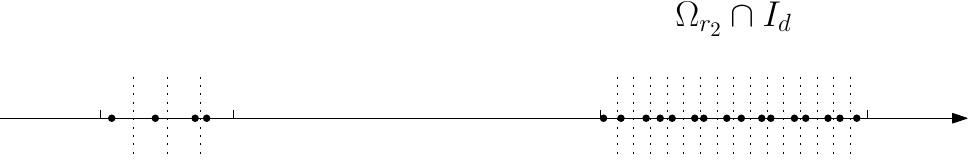}
\caption{The random set $\Omega_{\text{rand},\text{lac}} := \Omega_{r_2}$ contains almost surely uniformly distributed subset of arbitrarily large cardinal.}  
\end{figure}

We can now prove Proposition \ref{P3} since we have $$\mathbb{P}\left( \bigcap_{N \geq 1} B_N \right) = 1.$$ Precisely, for any $N \geq 1$, the event $B_N$ occurs almost surely and this means that there exists a (for each $N$, we just need one) dyadic interval $I_d$ such that $$ \Omega^{-1}_{\text{rand},\text{lac}} \cap I_{d,l} \neq \emptyset $$ for any $1 \leq l \leq 2^N$. We let $u_l$ be one point in $\Omega^{-1}_{\text{rand},\text{lac}} \cap I_{d,l}$ and we claim that the set $$U := \left\{ u_{2l} : 1 \leq l \leq 2^{N-1} \right\}$$ satisfy the condition of Proposition \ref{P3} with $\delta \simeq 2^{d-N}$. In particular, we have $G(U) \lesssim 1 $ for $U \subset \Omega^{-1}_{\text{rand},\text{lac}} $ with arbitrary large cardinal. This yields almost surely $$P(\Omega_{\text{rand},\text{lac}}) \lesssim 1$$ which concludes the proof.

{}


\begin{thebibliography}{}

\end{thebibliography}


\begin{thebibliography}{}




\bibitem{DANIELLOGAUVANMOONENS}
E. D'Aniello, A. Gauvan, L. Moonens, \emph{(Un)boundedness of directional maximal operators through a notion of "Perron capacity'' and an application}, 	accepted for publication in Proceedings of the AMS.

\bibitem{DM}
E. D’Aniello and L. Moonens, \emph{Differentiating Orlicz spaces with rect- angles having fixed shapes in a set of directions}, Z. Anal. Anwend., 39(4):461-473, 2020

\bibitem{BATEMAN}
M. D. Bateman, \emph{Kakeya sets and directional maximal operators in the plane}, Duke Math. J.
147:1, (2009), 55–77.

\bibitem{BATEMANKATZ}
M. D. Bateman and N.H. Katz, \emph{Kakeya sets in Cantor directions}, Math. Res. Lett. 15 (2008), 73–81.

\bibitem{CORDOBAFEFFERMAN II}
A. Cordoba, R. Fefferman, \emph{On differentiation of integrals}, Proc. Nat. Acad. Sci.
U.S.A. 74:6, (1977), 2211–2213.

\bibitem{GAUVAN}
A. Gauvan, \emph{Application of Perron trees to geometric maximal operators}, Colloq. Math. 172 (2023), no. 1, 1-13.

\bibitem{HareR}
K. Hare and J.-O. Rönning, \emph{Applications of generalized Perron trees to maximal functions and density bases}, J. Fourier Anal. and App. 4 (1998), 215–227.


\bibitem{KATZ}
N. H. Katz, \emph{A counterexample for maximal operators over a Cantor set of directions}, Math. Res. Lett., 3(4):527-536, 1996.

\bibitem{NSW}
A. Nagel, E. M. Stein, and S. Wainger, \emph{ Differentiation in lacunary directions}, Proc.
Nat. Acad. Sci. U.S.A. 75:3, (1978), 1060–1062.



\bibitem{KP2}
E. Kroc and M. Pramanik, \emph{Lacunarity, Kakeya-type sets and directional maximal operators}, arXiv:1404.6241.



\bibitem{FEFFERMAN}
C. Fefferman, \emph{The multiplier problem for the ball}, Annals of Mathematics, 94(2), (1971) 330-336.

\end{thebibliography}
\end{document}